\documentclass[12pt]{amsart}
\usepackage{graphicx,amssymb}
\usepackage{mathrsfs}
\usepackage{amssymb,amsmath,amsthm,color}
\usepackage{graphicx,mcite}
\usepackage{hyperref}
\usepackage{url}
\usepackage{setspace}

\theoremstyle{definition}

\theoremstyle{remark}

\numberwithin{equation}{section}
\newcommand{\R}{\mathbb R}

\def\TagOnRight

\def\R {\mathbb{R}}
\newcommand{\be}{\begin{equation}}
\newcommand{\ee}{\end{equation}}
\newcommand{\bea}{\begin{eqnarray}}
\newcommand{\eea}{\end{eqnarray}}
\newcommand{\Bea}{\begin{eqnarray*}}
\newcommand{\Eea}{\end{eqnarray*}}
\newcommand{\bt}{\begin{Theorem}}
\newcommand{\et}{\end{Theorem}}

\newcommand{\bpr}{\begin{Proposition}}
\newcommand{\epr}{\end{Proposition}}

\newcommand{\bl}{\begin{Lemma}}
\newcommand{\el}{\end{Lemma}}
\newcommand{\bi}{\begin{itemize}}
\newcommand{\ei}{\end{itemize}}

\newtheorem{Definition}{Definition}[section]
\newtheorem{Theorem}[Definition]{Theorem}
\newtheorem{Lemma}[Definition]{Lemma}
\newtheorem{Proposition}[Definition]{Proposition}
\newtheorem{Corollary}[Definition]{Corollary}
\newtheorem{Remark}[Definition]{Remark}

\begin{document}
\baselineskip16pt

\title{Maximal functions associated to flat plane curves with Mitigating factors}
\author{Ramesh Manna}
\address{School of Mathematics, Harish-Chandra Research Institute, Allahabad, India 211019.}
\address {\& Homi Bhabha National Institute, Training School Complex, Anushakti Nagar, Mumbai 400085, India.} 
\email {rameshmanna@hri.res.in, rameshmanna21@gmail.com}
\subjclass[2010]{Primary 42B25; Secondary 42B15; 46T30}
\date{\today}
\keywords{Maximal operator, infinitely flat curve, Mitigating factors, local smoothing.}

\begin{abstract}

We study the boundedness problem for maximal operators $\mathbb{M}_{\sigma}$ associated to flat plane curves with Mitigating factors, defined by
$$\mathbb{M}_{\sigma}f(x) \, := \, \sup_{1 \leq t \leq 2} \left|\int_{0}^{1} f(x-t\Gamma(s)) \, (\kappa(s))^{\sigma} \,  ds\right|,$$
where $\kappa(s)$ denotes the curvature of the curve $\Gamma(s)=(s, g(s)+1), ~g(s) \in C^5[0,1]$ in $\mathbb{R}^2$. Let $\triangle$ be the closed triangle with vertices $P=(\frac{2}{5}, \frac{1}{5}), ~ Q=(\frac{1}{2}, \frac{1}{2}), ~ R=(0, 0).$

In this paper, we prove that for $ (\frac{1}{p}, \frac{1}{q}) \in \left[(\frac{1}{p}, \frac{1}{q}) :(\frac{1}{p}, \frac{1}{q}) \in \triangle \setminus \{P, Q\} \right]  \cap \left[(\frac{1}{p}, \frac{1}{q}) :q > max\{\sigma^{-1},2\} \right]$,  there is a constant $B$ such that 
$ 
\|\mathbb{M}f\|_{L^q(\mathbb{R}^2)} \leq \, B \, \|f\|_{L^p(\mathbb{R}^2)}. $
\end{abstract}
\maketitle

\section{Introduction} \label{section1}
Let $\mathbb{C}$ denote a  smooth, compactly supported curve in the plane that does not pass through the origin, and denoting by $t \, \mathbb{C}$ the curve dilated by a factor $t>0,$ we consider the averaging operator defined for functions $f \in \mathscr{S},$ the Schwartz class of functions, by 
$$M_tf(x) \, := \, \int_{\mathbb{C}} f(x-ty) \, d\sigma(y),$$
where $d\sigma$ denotes the normalized Lebesgue measure over the curves $\mathbb{C}$. Consider now the maximal operator given by 
\begin{eqnarray}
\mathcal{M}f(x) \, := \, \sup_{t>0} |M_tf(x)|. \nonumber
\end{eqnarray}

It is not obvious that such averaging operators are well defined for $f$ in $L^p-$ spaces, since the curve $\mathbb{C}$ has measure zero in $\mathbb{R}^2.$ Nevertheless, a priori $L^p-L^q$ estimates are possible when $\mathbb{C}$ has suitable curvature properties.
Therefore, a natural question, we ask is for what range of the exponents $p$ and $q$ is  the following a priori inequality satisfied:
\begin{eqnarray} \label{1.1}
\|\mathcal{M}f\|_{L^q(\mathbb{R}^2)} \leq \, B \, \|f\|_{L^p(\mathbb{R}^2)}, ~ f \in \mathscr{S}.
\end{eqnarray}

%The aim of this paper is to study the boundedness of maximal operator associated with curves from $L^p(\mathbb{R}^2)$ to $L^q(\mathbb{R}^2).$ 
%There is a vast literature on maximal and averaging operators over families of lower dimensional curves in the plane. 
%The main issue turns out to be the curvature: roughly speaking, non-flat curves admit non-trivial maximal estimates; whereas flat curves do not.
%More generally, curvature condition plays a crucial role in the analysis of $\mathcal{M},$ and also in other problems in harmonic analysis, cf. \cite{EW}, \cite{EMS}. A fundamental and representative positive result in this direction is the Bourgain's circular maximal operator.

There is a vast literature on maximal and averaging operators over families of lower dimensional curves in the plane (see \cite{EW, EMS}). The study of such a maximal operator over dilations of a fixed curves $\mathbb{C} \subset \mathbb{R}^2$ has its beginnings in the circular maximal theorem of Bourgain (see, \cite{B}).
Bourgain showed that when $ \mathbb{C}=\mathcal{S}^1,$ the unit circle, the corresponding maximal operator is bounded on $L^p(\mathbb{R}^2)$ for $p > 2.$
His proof of the circular maximal theorem  relies more directly on the geometry involved. The relevant geometry information concerns intersections of pairs of thin annuli, (for more details, see \cite{B}). 
Other proof is due to Mockenhaupt, Seeger and Sogge (see, \cite{MSS2}) and proof of this result is based on their local smoothing estimates (see also, \cite{MSS}).
These local smoothing estimates, as well as Bourgain's original techniques actually implies that if one modifies the definition so that the supremum is taken over $1 < t < 2,$
then the resulting circular maximal operator  is bounded from $L^p(\mathbb{R}^2)$ to $L^q(\mathbb{R}^2)$ for some $q > p.$
Here, the maximal operator $\mathbb{M}$ is defined by 
\begin{eqnarray} \label{1.2}
\mathbb{M}f(x) \, := \, \sup_{1 \leq t \leq 2} \left|\int_{\mathbb{C}} f(x-ty) \,  d\sigma(y) \right|.
\end{eqnarray}

Let $\triangle$ be the closed triangle with vertices $P=(\frac{2}{5}, \frac{1}{5}), ~ Q=(\frac{1}{2}, \frac{1}{2}), ~ R=(0, 0).$ In 1997, Schlag (see, \cite{SG}) showed that if $\mathbb{C}$ is unit circle, then the maximal operator $\mathbb{M}$ satisfies the inequality \eqref{1.1} if $(\frac{1}{p}, \frac{1}{q})$ lies in the interior of $\triangle.$ His result was obtained using the "combinatorial method" of Kolasa and Wolff cf. \cite{KW}. A different proof of this result was later obtained by Schlag and Sogge cf. \cite{SGS}, which was based on a simple application of Sobolev's theorem and the appropriate local smoothing estimates. Schlag also showed that except possibly for endpoints, this result is sharp (see, \cite{SG}, \cite{SGS}). Later, in 2002, Sanghyuk Lee (see, \cite{SL}) consider the remaining endpoint estimates for the circular maximal operator.

Therefore, the main question, we ask is, whether the same  priori maximal inequality \eqref{1.1} holds even if we consider a situation when the curvature is allowed to vanish of finite order on a finite set of isolated points. In this connections, Iosevich (see, \cite{IOSE}) had already shown that: If $\mathcal{C}$ is of finite-type curve, whose curvature vanishes to order at most $m-2$ at a single point. Then, the inequality 
$$\|\mathcal{M}f\|_{L^p(\mathbb{R}^2)} \leq \, B_p \, \|f\|_{L^p(\mathbb{R}^2)},$$ holds for $p > m,$ also this result is sharp i.e., $\mathcal{M}$ is unbounded if $p=m$. If we study finite-type curves in the plane given by $\Gamma(s)=(s, g(s)+1),~s \in[0,1],$ for some suitably smooth $g$, where $0=g(0)=g'(0)=\dots=g^{(m-1)}(0)\neq g^{(m)}(0)>0,$ then we can reinterpret his results as follows. 
Define 
$$\mathcal{M}f(x) \, := \, \sup_{t>0} \left|\int_{2^{-k}}^{2^{1-k}} f(x-t\Gamma(s)) \,  ds\right|,$$ for Schwartz class of functions $f.$ Iosevich proved that
$$\|\mathcal{M}f\|_{L^p\rightarrow L^p} \leq c_p  \, 2^{-k(1-m/p)},$$ for $p >2.$ If we note that $\kappa(s),$ the curvature of the curve $\Gamma(s)$ is approximately $2^{-k(m-2)}$ whenever $s \in [2^{-k}, 2^{1-k}],$ then we have the operator 
$$\mathcal{M}_{\sigma}f(x) \, := \, \sup_{t>0} \left|\int_{0}^{1} f(x-t\Gamma(s)) \, (\kappa(s))^{\sigma} \,  ds\right|,$$ is bounded on $L^p$ for some $p>2,$ if $\sigma$ is sufficiently large, since 
\begin{eqnarray}
\|\mathcal{M}_{\sigma}f\|_{L^p\rightarrow L^p} &\leq& C \sum_{k \geq 0} 2^{-k(m-2) \sigma} \, \|\mathcal{M}_kf\|_{L^p\rightarrow L^p} \nonumber \\
&\leq& C \sum_{k \geq 0} 2^{-k((m-2) \sigma+ 1-m/p)} 
\end{eqnarray}

which is finite so long as $\sigma > (m/p-1) (m-2)^{-1}$. If we want to choose $\sigma$ independent of $m>2,$ the type of the curve, such that $\mathcal{M}_{\sigma}$ is bounded on $L^p$ for some fixed $p> 2$, then clearly we can take $\sigma =1/p$. In this connection, Marletta \cite{GM} proved that $\mathcal{M}_{\sigma}$ is bounded on $L^p$ for $p> \max\{\sigma^{-1},2\}$, for a class of infinitely flat, convex curves in the plane. Counterexamples in \cite{GM} showed that this is the best possible result, in the sense that there exist flat curves for which $\mathcal{M}_{\sigma}$ is unbounded for $2 < p \leq \sigma^{-1}.$

\begin{Remark}
As has been observed by J. G. Bak, such estimates for maximal functions with
mitigating factors give rise to Orlicz space estimates for the corresponding maximal functions with no mitigating factor, with the help of Bak's interpolation lemma. See, for example \cite{BAK} for the higher dimensional equivalents.
\end{Remark}

In this paper, we consider the maximal operator $\mathbb{M}_{\sigma},$ given by 
\begin{eqnarray} \label{01.3}
\mathbb{M}_{\sigma}f(x) \, := \, \sup_{1 \leq t \leq 2} \left|\int_{0}^{1} f(x-t\Gamma(s)) \, (\kappa(s))^{\sigma} \,  ds\right|.
\end{eqnarray}
Here, we shall extend the result of Marletta \cite{GM} to $L^p-L^q$ estimates for the corresponding maximal operator \ref{01.3} associated to families of smooth, compactly supported curves in the plane. The proof of our main result, Theorem \ref{thm2.1} will strongly make use of the results in \cite{GM} by Marletta.

We first decompose the operator $\mathbb{M}_{\sigma}$ into a family of operators $\{M_k\}_{k \geq 0}.$ We can actually obtain the $L^p \to L^q$ estimate from the corresponding $L^p-L^q$ estimate for each of the operators $M_k,~k \geq 0$, by summing a geometric series. We make a decomposition of the $s$- space in intervals where the curvature $\kappa(s)$ of $\Gamma(s)$ is relatively constant, see section 4. This is markedly different from the decomposition used in Fourier analysis (e.r. averages over flat curves in the plane), where dealing with dyadic decomposition in the space variable yields required estimates for the maximal operator over finite type curves \cite{IOSE}.

Our estimate for $M_kf$ relies on stationary phase method and the local smoothing estimates of S. Lee \cite{SL}. The $L^p$- boundedness of the operators $M_kf, ~k\geq 0$ are known by the work of Iosevich. Here, we will extend this estimate to $L^p \to L^q$ and also in the case where we consider the maximal operator over infinitely flat curves.

A key fact which leads to the proof is that the norm of the operator $M_k$ will remain same even the curve $\Gamma(s)$ is replaced by $L(\Gamma(s))$, where $L$ is the invertible linear map on $\mathbb{R}^2$, see Lemma \ref{Lma4.1}.

Next, we formulate the main result of this paper in $(\S, 2)$. 

\section{main result and Preliminaries}
In this section, we shall state our main result of this paper.
\begin{Theorem} \label{T03.1}
Let $\Gamma(s)=(s, g(s)+1), ~s \in [0,1],$ be a curve in $\mathbb{R}^2$ such that $g \in C^5[0,1], ~g(0)=g'(0)=0,~g^{(r)}(s)$ is single-signed and monotonic for $r=1,2,3,4,5,$ and both $\log~ g^{''}(s)$ and $\log~ g^{'''}(s)$ are concave. 
Let, $$\mathbb{M}_{\sigma}f(x) \, = \, \sup_{1 \leq t \leq 2} \left|\int_{0}^{1} f(x-t\Gamma(s)) \, (\kappa(s))^{\sigma} \,  ds \right|,$$
where $\kappa(s)$ denotes the curvature of the curve $\Gamma(s).$

Then, for $ (\frac{1}{p}, \frac{1}{q}) \in \left[(\frac{1}{p}, \frac{1}{q}) :(\frac{1}{p}, \frac{1}{q}) \in \triangle \setminus \{P, Q\} \right]  \cap \left[(\frac{1}{p}, \frac{1}{q}) :q > max\{\sigma^{-1},2\} \right]$,  there is a constant $B$ such that the following inequality 
\begin{eqnarray} \label{thm2.1}
\|\mathbb{M}_{\sigma}f\|_{L^q(\mathbb{R}^2)} \leq \, B  \,  \|f\|_{L^p(\mathbb{R}^2)}, ~ f \in \mathscr{S}(\mathbb{R}^2)
\end{eqnarray}
 holds.
  
\end{Theorem}

\begin{Remark}
As $|\Gamma'(s)|$ is bounded between two constants $c$ and $C$ independent of everything, we have
$$0 \leq c \, g^{''}(s) \leq \kappa(s) \leq C \, g{''}(s),$$ 
and so we may replace $\kappa(s)$ by $g^{''}(s)$ in the maximal function. Notice also that as $g^{''}(s)$ is bounded, it suffices to consider only $0 < c \leq g^{''}(s) \leq 1.$ The strange conditions concerning $\log g^{''}(s)$ and $\log g^{'''}(s)$ are more than we require here, see section $4$.
\end{Remark}

We now detail the dyadic decomposition of the dual space that is needed. 
\subsection{\large The dyadic decomposition} \label{subs1.1}
The proof of our main result, as well as many other arguments that involve explicitly (or implicitly) the Fourier transform, makes use of the division of the dual (frequency) space into dyadic shells. Dyadic decomposition, whose ideas originated in the work of Littlewood and Paley, and others, will now be described in the form most suitable for us (see, \cite{EMS}).

Let $\beta$ be a non negative radial function in $C_c^{\infty}(\mathbb{R}^2)$ supported in $\{ \frac{1}{2} \leq |\xi| \leq 2 \}$ such that 
$$\sum_{j= - \infty}^{\infty} \, \beta(2^{-j} \xi) =1 \mbox{ for } \xi \neq 0.$$
For example, we shall take,
\begin{eqnarray*}
\phi(\xi)=
\left\{\begin{array}{ll}
1, &\mbox{ if } |\xi| \leq \frac{1}{2}\\
~~0, &\mbox{ if } |\xi| \geq 1.
\end{array}\right.
\end{eqnarray*}
and $$\beta(\xi) \, = \, \Phi(\frac{\xi}{2})-\Phi(\xi).$$
Then, one can easily see that  $\sum\limits_j\beta(2^{-j}\xi) \, = \, 1, ~~~\xi \, \neq \, 0$ (see \cite{D}).

We shall use $C$ as a constant independent of $j,$ in several times without mention it.

\section{Scaling}
Lemma \ref{Lma2} is crucial in the proof of our theorem. The idea of the proof of the lemma \ref{Lma2} is based on Iosevich's approach in \cite{IOSE}.

\begin{Lemma} \label{Lma2}
Let $\tilde{\Omega}(s)=(s+\alpha, \omega(s)+\beta), ~s \in [0,1],$ be a curve in $\mathbb{R}^2$ such that $\omega \in C^5[0,1], ~\omega(0)=\omega'(0)=0,~\omega'(s) \leq 1/2$ and $0 < c \leq g^{''} \leq C < \infty,$ where $\beta$ and  $\alpha$ are constants satisfying $|\beta| \geq max \{|\alpha|, 2\}.$  
Let, $$\tilde{M}_tf(x) \, =  \int_{0}^{1} f(x-t \tilde{\Omega}(s)) \,  ds,$$ and

$$\tilde{\mathbb{M}}f(x) \, = \, \sup_{1 \leq t \leq 2} |\tilde{M}_tf(x)|,$$ be the averaging operator and the maximal function corresponding to $\tilde{\Omega},$ respectively.

Then, for $ (\frac{1}{p}, \frac{1}{q}) \in \left[(\frac{1}{p}, \frac{1}{q}) :(\frac{1}{p}, \frac{1}{q}) \in \triangle \setminus \{P, Q\} \right]  \cap \left[(\frac{1}{p}, \frac{1}{q}) :q > 2 \right]$,  the following inequality 
\begin{eqnarray} \label{thm3.1}
\|\tilde{\mathbb{M}}f\|_{L^q(\mathbb{R}^2)} \leq \, B  \, |\beta|^{1/q} \,  \|f\|_{L^p(\mathbb{R}^2)}, ~ f \in \mathscr{S}(\mathbb{R}^2)
\end{eqnarray}
 holds, where the constant $B$ depends only on the $c, C$ and the $C^5$ norm of $\omega(s)$.  
\end{Lemma}

\begin{Remark}
In view of Iosevich's theorem  (see, \cite{IOSE}), the maximal operator $\mathbb{M}$ is also of course bounded when the exponents lie on the half open line connecting $(\frac{1}{2}, \frac{1}{2})$ and $(0,0)$.
\end{Remark}

\begin{proof}
Our proof will consist of three main steps. First we shall decompose each operator $\tilde{M}_t$ away from the flat point. Then we shall use the method of stationary phase to express each dyadic operator in terms of the Fourier transform of the surface measure on each dyadic piece.  We shall then use a scaling argument and  a technical lemma to reduce the problem to the local smoothing estimates (see, S. Lee, \cite{SL, CDS2}) for the corresponding Fourier integral operator.

We now turn to the details. Now, choose a bump function $\phi \in C_c^{\infty}(\mathbb{R})$ supported in $[-2,2]$. Consider the linear operator
$$\tilde{M}_tf(x) \, =  \int_{\mathbb{R}} f(x-t \tilde{\Omega}(s)) \, \phi(s) \,  ds,$$

Then, the inequality \eqref{thm3.1} is equivalent to the following estimate for $\tilde{\mathbb{M}}:$
\begin{eqnarray} \label{04.3}
\|\tilde{\mathbb{M}}f\|_{L^q(\mathbb{R}^2)} \leq \, B \, (|\sigma|+1)^{\frac{1}{q}} \,  \|f\|_{L^p(\mathbb{R}^2)}, ~ f \in \mathscr{S}(\mathbb{R}^2),
\end{eqnarray}

for every $\sigma=(\alpha, \beta) \in \mathbb{R}^2,$ where $B$ is an admissible constant.
\noindent

By means of the Fourier inversion formula, we can write
$$\tilde{M}_tf(x)= \frac{1}{(2 \pi)^2} \int_{\mathbb{R}^2} e^{i(x-t\sigma). \xi} \, H(t \xi) \, \hat{f}(\xi) \, d \xi,$$
where
$$H(\xi_1, \xi_2): = \int_{\mathbb{R}} e^{-i(\xi_1 s+ \xi_2 \omega(s))} \, \phi(s) \, ds.$$
Using the method of stationary phase to $H(\xi),$ we obtain
$$H(\xi)=e^{iq(\xi)} \frac{\chi(\frac{\xi_1}{\xi_2}) A(\xi)}{(1 + |\xi|)^{\frac{1}{2}}} + B(\xi),$$
where $\chi$ is a smooth function supported on a small neighborhood of the origin. Moreover, $q(\xi)$ is a smooth function of $\xi$  which is homogeneous of degree $1$ in $\xi$ and also the Hessian $D^2_{\xi}q(\xi)$ has rank $1$.  Moreover, $A$ is a symbol of order zero such that $A(\xi)=0,$ if $|\xi| \leq C,$ and 
\begin{eqnarray}
|\xi^{\alpha} \, D^{\alpha}_{\xi} A(\xi)| \leq C_{\alpha}, ~ \alpha \in \mathbb{N}^2, ~ |\alpha| \leq 3,
\end{eqnarray}

where the $C_{\alpha}$ are admissible constants. Finally, $B$ is a remainder term satisfying 
\begin{eqnarray} \label{04.5}
| D^{\alpha}_{\xi} B(\xi)| \leq C_{\alpha, N} \, (1+|\xi|)^{-N},  ~ |\alpha| \leq 3,~ 0 \leq N \leq 3,
\end{eqnarray} 

again with admissible constants $C_{\alpha, N}$.
If we put 
$$\tilde{M}^0_tf(x)= \frac{1}{(2 \pi)^2} \int_{\mathbb{R}^2} e^{i(x-t\sigma). \xi} \, B(t \xi) \, \hat{f}(\xi) \, d \xi,$$
then by \eqref{04.5} $\tilde{M}^0_tf(x) = f \star k_t^{\sigma}(x),$ where $k_t^{\sigma}(x)=t^{-2}k(\frac{x}{t})$ and where $k^{\sigma}$ is the translate 
\begin{eqnarray} \label{03.5}
k^{\sigma}(x): = k(x-\sigma)
\end{eqnarray}

of $k$ by the vector $\sigma$ of a fixed function $k$ satisfying an estimate of the form
\begin{eqnarray} \label{03.7}
|k(x)| \leq C (1 + |x|)^{-3}.
\end{eqnarray}

Let $\tilde{\mathbb{M}}^0 f(x):= \sup_{1 \leq t \leq 2} |\tilde{M}^0_tf(x)|$ denote the corresponding maximal operator. Then,
\eqref{03.5} and the inequality \eqref{03.7} show that $\|\tilde{\mathbb{M}}^0 \|_{L^{\infty} \rightarrow L^{\infty}} \leq C,$ with a constant $C$ which does not depend on $\sigma.$ Moreover, scaling by the factor $(|\sigma|+1)^{-1}$ in direction of the vector $\sigma,$ we see that
$$|\tilde{M}^0_tf(x)| \lesssim (|\sigma|+1) \, M,$$
where $M$ is the Hardy-Littlewood maximal operator. Hence, \eqref{04.3} holds for 
$\tilde{\mathbb{M}}^0$ in place of $\tilde{\mathbb{M}},$ for every $1 \leq p \leq q.$

The maximal operator $\tilde{\mathbb{M}}^1,$ corresponding to the family of averaging operators
$$\tilde{M}^1_tf(x)= \frac{1}{(2 \pi)^2} \int_{\mathbb{R}^2} e^{i[\xi.x-t(\sigma.\xi+q(\xi))]} \,  \frac{\chi(\frac{\xi_1}{\xi_2}) A(\xi)}{(1 + |\xi|)^{\frac{1}{2}}} \, \hat{f}(\xi) \, d \xi,$$
remains to be studied.

Hence it is enough to show that 
\begin{eqnarray} \label{mainest}
\|\sup_{1 \leq t \leq 2}|\tilde{M}^1_tf(x)|\|_{L^q(\mathbb{R}^2)} \, \leq \, B \, (|\sigma|+1)^{\frac{1}{q}} \, \|f\|_{L^p(\mathbb{R}^2)}.
\end{eqnarray}

Using this Fourier integral representation, we shall break up the operators dyadically. For this purpose,
let us fix $\beta \in C_c^\infty(\mathbb{R}\setminus0)$ satisfying $\sum\limits_{-\infty}^{\infty}\beta(2^{-j}s)=1,~ s \neq 0,$ as in $(\S, 2).$

We then define the dyadic operator $A_{j,t}$ by
$$A_{j,t}f(x,t)= \frac{1}{(2 \pi)^2} \int_{\mathbb{R}^2} e^{i[\xi.x-t(\sigma.\xi+q(\xi))]} \,  \frac{\chi(\frac{\xi_1}{\xi_2}) A(t \xi)}{(1 + t |\xi|)^{\frac{1}{2}}} \, \beta(2^{-j} |t\xi|)\, \hat{f}(\xi) \, d \xi.$$

Since we may assume that $A$ vanishes on a sufficiently large neighborhood of the origin, we have $A_{j,t}f =0,$ if $j \leq0,$ so that
$$\tilde{\mathbb{M}}^1_tf(x) = \sum_{j=1}^{\infty} A_{j,t}f(x).$$ 

Therefore, the inequality \eqref{04.3}, would follow from showing that 
when  $ (\frac{1}{p}, \frac{1}{q}) \in \left[(\frac{1}{p}, \frac{1}{q}) :(\frac{1}{p}, \frac{1}{q}) \in \triangle \setminus \{P, Q\} \right]  \cap \left[(\frac{1}{p}, \frac{1}{q}) :q > 2 \right]$, there is a constant $B$ such that 
\begin{eqnarray} \label{3.5}
\|\sup\limits_{1 \leq t \leq 2}\sum_{j=1}^{\infty}|A_{j,t}f(x)|\|_{L^q(\mathbb{R}^2)} \leq B \, (|\sigma|+1)^{\frac{1}{q}} \, \|f\|_{L^p(\mathbb{R}^2)}.
\end{eqnarray}

Now, choose a bump function $\psi \in C_c^{\infty}(\mathbb{R})$ supported in $[\frac{1}{2}, 4]$ such that $\psi(t)=1$ if $1 \leq t \leq2.$

In order to estimate \eqref{3.5}, we use the following well- known estimate (see e.g., \cite{IOSE}, Lemma 1.3),

\begin{eqnarray} \label{3.6'}
&&\sup_{t \in \mathbb{R}}|\psi(t) \, A_{j,t}f(x)|^q  \\
&\leq& q\left(\int_{-\infty}^{\infty} |\psi(t) \, A_{j,t} f(x)|^q \, dt \right)^{\frac{q-1}{q}} \, \left( \int_{-\infty}^{\infty} \left|\frac{\partial}{\partial t}\left(\psi(t) \, A_{j,t}f(x) \right)\right|^q \, dt \right)^{\frac{1}{q}} \nonumber
\end{eqnarray}

which follows by using the fundamental theorem of calculus and H$\ddot{o}$lders inequality.

By H$\ddot{o}$lder's inequality, this implies
\begin{eqnarray} \label{3.6}
&&\|\sup_{1 \leq t \leq 2}  A_{j,t}f(x)\|^q_{L^q(\mathbb{R}^2)}  \\
&\leq& C \, q \left(\int_{\frac{1}{2}}^{4} \int_{\mathbb{R}^2} |A_{j,t} f(x)|^q \, dx \,  dt \right)^{\frac{q-1}{q}} \, \left( \int_{\frac{1}{2}}^{4} \int_{\mathbb{R}^2} \left|\frac{\partial}{\partial t}\left(A_{j,t} f(x) \right)\right|^q \,  dx \, dt \right)^{\frac{1}{q}} \nonumber \\
&& + C \int_{\frac{1}{2}}^{4} \int_{\mathbb{R}^2} | A_{j,t}f(x)|^q \, dx \,  dt. \nonumber
\end{eqnarray}

Now, $$\frac{\partial}{\partial t}\left(A_{j,t}f(x) \right) \, = \, \frac{1}{(2 \pi)^2} \,\int_{\mathbb{R}^2} e^{i[\xi.x-t(\sigma.\xi+q(\xi))]} \,  \chi(\frac{\xi_1}{\xi_2}) \, h(t,j, \xi)  \, \hat{f}(\xi) \, d \xi,$$
where, 
\begin{eqnarray}
h(t, j, \xi) \, &=& \, -i\frac{[\sigma. \xi + q(\xi)] \, A(t \xi)}{(1 + t |\xi|)^{\frac{1}{2}}} \, \beta(2^{-j} |t\xi|) + \frac{\partial}{\partial t} \,  \left( \frac{A(t \xi)}{(1 + t |\xi|)^{\frac{1}{2}}} \right) \, \beta(2^{-j} |t\xi|) \nonumber  \\&+& \frac{A(t \xi)}{(1 + t |\xi|)^{\frac{1}{2}}} \,  (2^{-j} |\xi|) \, \beta'(2^{-j} |t\xi|). \nonumber 
\end{eqnarray}

Now, if $t\sim1,$ since $A$ vanishes near the origin, we see that the amplitude of $A_{j,t}$ can be written as $2^{-j/2} \, a_{j,t}(\xi),$ where $a_{j,t}$ is a symbol of order $0$ localized where $|\xi|\sim2^j.$ Similarly, the amplitude of $\frac{\partial}{\partial t}A_{j,t}$ can be written as $2^{j/2}(|\sigma|+1) \, b_{j,t},$ where $b_{j,t}$ is a symbol of order $0$ localized where $|\xi|\sim2^j.$

Since, $q(\xi) \approx |\xi| \approx 2^j, $ on the support of $\beta,$ we can calculate the orders of the symbols to see that $\|\sup_{1 \leq t \leq 2}  A_{j,t}f(x)\|_{L^q(\mathbb{R}^2)}$ in \eqref{3.6} is dominated by 
$$C \, 2^{-j(\frac{1}{2}- \frac{1}{q})} \, (|\sigma|+1)^{1/q} \,  \|\mathcal{F}_{ j}f\|_{L^q(\mathbb{R}^3)},$$
where $$\mathcal{F}_{ j}f \, = \, \psi(t) \, \int e^{i<x, \xi>} \, e^{-it q(\xi)} \, \beta(2^{-j} |\xi|) \, a(t, \xi) \, \hat{f}(\xi) \, d\xi, $$ and where $a(t, \xi)$ is a symbol of order $0$ in $\xi.$ A similar argument can also be seen in \cite{MIK}.

Now, we need a local smoothing estimates for the operators of the form 
\begin{eqnarray} \label{3.9}
P_jf(x,t) \, = \,  \int e^{i<x, \xi>} \, e^{it q(\xi)} \,a(t, \xi) \, \beta(2^{-j}|\xi|) \, \hat{f}(\xi) \, d\xi,
\end{eqnarray}

where $a(t, \xi)$ is a symbol of order $0$ in $\xi$ and the Hessian matrix of $q$ has rank $1$ everywhere.

We will get the smoothing estimates by using some sharp Carleson-Sj$\ddot{o}$lin type estimates for the $2-$dimensional wave equation.
For this, let us define, 
$$\mathcal{F}_tf(x) \, = \, \int_{\mathbb{R}^2} e^{i[x.\xi +t |\xi|]} \, \hat{f}(\xi) \, d\xi.$$ 

We wish to use the following $L^p-L^q$ local smoothing estimates (see, \cite{SL}), which we shall only use for $N=2^j.$

\begin{Theorem} \label{thm3.5}
If supp $\hat{f} \subset \{ \xi \in \mathbb{R}^2 : |\xi|\sim N\},$ then for $\frac{1}{p} + \frac{3}{q} \, = \, 1, ~ \frac{14}{3} < q \leq \infty,$
we have

\bea \label{lee1}
\left(\int_{\mathbb{R}^2} \int_1^2 |\mathcal{F}_tf(x)|^q \, dt \, dx\right)^{\frac{1}{q}} \, \leq \, C \, N^{\frac{3}{2}- \frac{6}{q}} \, \|f\|_{L^p(\mathbb{R}^2)}.
\eea

\end{Theorem} 

Note that Sogge and Schlag (see, Theorem 1.1, \cite{SGS}) proved the smoothing estimate \eqref{lee1} for $q \geq 5$ and up to an endpoint, bounds in \eqref{lee1} are of the best possible nature. In the same paper, they also consider the Fourier integral operator of the form 
$$F^{\phi}f(x,t)=\int e^{i \phi(x,t, \xi)} a(t,x, \xi) \, \hat{f}(\xi) \, d\xi,$$

where, $a \in C^{\infty}([1,2] \times \mathbb{R}^2 \times \mathbb{R}^2)$ vanishes for $x$ outside of a fixed compact set and satisfies 
$$|D_{t,x}^{\gamma_1} \, D_{\xi}^{\gamma_2} a(t,x, \xi)| \leq C_{\gamma}(1+|\xi|)^{-|\gamma_2|}.$$ Also, the phase functions are real, in $C^{\infty}([1,2] \times \mathbb{R}^2 \times \mathbb{R}^2 \setminus 0)$ and homogeneous of degree one in $\xi.$

Also, the above phase function $\phi$ satisfies the conditions
\begin{eqnarray} \label{4.10}
 det \frac{\partial^2 \phi}{\partial x \partial \xi} \neq 0 
 \end{eqnarray} on supp of $a$.
 
\begin{eqnarray} \label{4.11}
&& \frac{\partial \phi}{\partial t}= q(t, x, \phi_x^{'}), \,  \mbox{ Corank } \, q_{\xi \xi}^{"} =1, 
\end{eqnarray}
on supp of $a$.

Under these hypothesis, they have proved the following local smoothing estimate for $q\geq 5$ and  supp $\hat{f} \subset \{ \xi \in \mathbb{R}^2 : |\xi|\sim N\},$ with $\frac{1}{p} + \frac{3}{q} \, = \, 1$,
\bea \label{lee2}
\left(\int_{\mathbb{R}^2} \int_1^2 |F^{\phi}f(x,t)|^q \, dt \, dx\right)^{\frac{1}{q}} \, \leq \, C \, N^{\frac{3}{2}- \frac{6}{q}} \, \|f\|_{L^p(\mathbb{R}^2)}.
\eea

Now, for our operator $P_jf(x,t),$  we would like the get the following  local smoothing estimate for $\frac{1}{p} + \frac{3}{q} \, = \, 1, ~ \frac{14}{3} < q \leq \infty$,
\begin{eqnarray} \label{3.8}
\left(\int_{\mathbb{R}^2} \int_1^2 |P_jf(x)|^q \, dt \, dx\right)^{\frac{1}{q}} \, \leq \, C \, 2^{j(\frac{3}{2}- \frac{6}{q})} \, \|f\|_{L^p(\mathbb{R}^2)}.
\end{eqnarray}
 
If we use the Sogge and Schlag's local smoothing estimate \eqref{lee2}, we get the smoothing estimate \eqref{3.8} for $q \geq 5.$

To prove the smoothing estimate \eqref{3.8} for $q>14/3$, we first observe that
the proof of the Theorem \eqref{thm3.5} for Fourier integral operator with phase function $x \cdot \xi+t |\xi|$, they have used the bilinear cone restriction estimate of Wolff \cite{W}, and Tao \cite{T} together with a decomposition technique which was used in (\cite{TV2}, section 4). 

Since $q$ is real smooth function away from the origin and homogeneous of degree one, there are positive constant $d$ and $D$ so that 
\bea \label{impp}
d|\xi| \leq q(\xi) \leq D |\xi|.
\eea 

Also, in 2006,  S. Lee (see Theorem 1.2, \cite{SL1}) proved the bilinear cone restriction estimate for the cone $F ={(\xi, \tau ) \in R^3:     q( \xi )=\tau, 1 \leq  \tau  \leq  2}$, where $q$ is a smooth, homogeneous function of degree one in $\xi$ and the Hessian of $q$ has rank one. Now, appealing to the bilinear cone restriction estimate of S. Lee and by \eqref{impp}, we get the  local smoothing estimate \eqref{3.8}  (as can be seen by simple modification of the proof of proposition 1.2 in \cite{SL}) for $\frac{1}{p} + \frac{3}{q} \, = \, 1, ~ \frac{14}{3} < q \leq \infty$.

Now, we comeback to prove the estimate \eqref{mainest}. Let, $$\mathbb{M}_jf(x) \, = \, \sup_{1 \leq t \leq 2} |A_{j,t}f(x)|.$$
Now, using the local smoothing estimates \eqref{3.8}, from \eqref{3.6}, it is easy to see that for $\frac{1}{p} + \frac{3}{q} \, = \, 1,$ and $ \frac{14}{3} < q \leq \infty,$
\begin{eqnarray} \label{3.10}
\|\mathbb{M}_jf\|_q \leq \, C \, 2^{j(1- \frac{5}{q})} \, (|\sigma|+1)^{1/q} \, \|f\|_{L^p(\mathbb{R}^2)}.
\end{eqnarray}

By Plancherel's theorem and from the estimate \eqref{3.6'}, it is easy to see that for $j \geq 1,$

\begin{eqnarray} \label{3.11}
\|\mathbb{M}_jf\|_2 \leq \, C \, (|\sigma|+1)^{1/2} \,  \|f\|_{L^2(\mathbb{R}^2)}.
\end{eqnarray}

A complex interpolation between \eqref{3.10} and \eqref{3.11} shows that if $(\frac{1}{p}, \frac{1}{q})$ is contained in the closed triangle with vertices $(1, 0),~(\frac{5}{14}, \frac{3}{14}),~(\frac{1}{2}, \frac{1}{2})$ but is not on the closed line segment $[(\frac{5}{14}, \frac{3}{14}),~(\frac{1}{2}, \frac{1}{2})],$ then 

\begin{eqnarray} \label{3.12}
\|\mathbb{M}_jf\|_q \leq \, C \, (|\sigma|+1)^{1/q} \,  2^{j\frac{(\frac{3}{p}- \frac{1}{q} -1)}{2}} \, \|f\|_{L^p(\mathbb{R}^2)}.
\end{eqnarray}

To sum up the last estimates, we use the following interpolation lemma. An explicit statement and proof of the lemma can be found in \cite{SL}, (Lemma $2.6$). Now, we denote by $L^{p, r}$ the Lorentz spaces.

\begin{Lemma} \label{L3.7} (An interpolation lemma)\\
Let $\epsilon_1, ~ \epsilon_2 > 0. $ Suppose that $\{T_j\}$ is a sequence of linear (or sublinear) operators such that for some $1 \leq p_1, ~ p_2 < \infty,$ and $1 \leq q_1, ~q_2 < \infty,$
$$\|T_jf\|_{q_1} \leq M_1 2^{\epsilon_1 j} \, \|f\|_{p_1}, ~\|T_jf\|_{q_2} \leq M_2 2^{-\epsilon_2 j} \, \|f\|_{p_2}. $$
Then $T= \sum T_j$ is bounded from $L^{p, 1}$ to $L^{q, \infty}$ with 
$$\|Tf\|_{L^{q, \infty}} \leq C \, M_1^{\theta} \, M_2^{1- \theta} \, \|f\|_{L^{p,1}}, $$ where $\theta=\frac{\epsilon_2}{(\epsilon_1
+\epsilon_2)},~ \frac{1}{q}=\frac{\theta}{q_1}+\frac{(1-\theta)}{q_2}, ~\frac{1}{p}=\frac{\theta}{p_1}+\frac{(1-\theta)}{p_2}.$
\end{Lemma} 

Using \eqref{3.12} and Lemma \ref{L3.7}, we have for $(\frac{1}{p}, \frac{1}{q}) \in [P, Q),$
\begin{eqnarray} \label{3.13}
\|\tilde{\mathbb{M}}^1f\|_{L^{q, \infty}} \leq \, C \, (|\sigma|+1)^{1/q} \, \|f\|_{L^{p,1}}.
\end{eqnarray}

Since $\mathbb{M}$ is a local operator, an interpolation (real interpolation) between these estimates and the trivial $L^{\infty}- L^{\infty}$ estimate, we get,  
 for $ (\frac{1}{p}, \frac{1}{q}) \in  \triangle \setminus \{P, Q\}$,
 \begin{eqnarray} \label{3.14}
 \|\tilde{\mathbb{M}}^1f\|_{L^{q}} \leq \, C \,(|\sigma|+1)^{1/q} \, \|f\|_{L^{p}}.
 \end{eqnarray}
 
For $ (\frac{1}{p}, \frac{1}{q}) \in \left[(\frac{1}{p}, \frac{1}{q}) :(\frac{1}{p}, \frac{1}{q}) \in  \triangle \setminus \{P, Q\} \right]  \cap \left[(\frac{1}{p}, \frac{1}{q}) :q > 2 \right],$ we thus get,
$$\|\tilde{\mathbb{M}}^1f\|_{L^q(\mathbb{R}^2)} \leq \, B \, (|\sigma|+1)^{1/q} \,  \|f\|_{L^p(\mathbb{R}^2)}.$$
 
Hence, we finish our proof of the lemma.
\end{proof}

We shall need one more lemma in our proof of Theorem \ref{T03.1}.
\begin{Lemma} \label{Lma4.1}
Let $L$ be an invertible linear map from $\mathbb{R}^2$ to itself with $det(J_L)=1,~ J_L=$ Jacobian matrix of $L,$ and let $\Gamma(s)$ and $\tilde{\Gamma(s)}$ be two curves in the plane, related by $\tilde{\Gamma(s)}=L(\Gamma(s)).$ Then the following two maximal functions have identical $L^p-L^q$ operator norms:
$$Mf(x)=\left|\sup_{1 \leq t \leq 2} \int f(x-t\Gamma(s)) \, ds\right|$$
$$\tilde{M}f(x)=\left|\sup_{1 \leq t \leq 2} \int f(x-t\tilde{\Gamma(s)}) \, ds\right|$$
\end{Lemma}

\begin{proof}
We see that,
\begin{eqnarray}
&&\tilde{M}f(x) = \left|\sup_{1 \leq t \leq 2} \int f(x-t\tilde{\Gamma(s)}) \, ds\right| \nonumber \\
&=&\left|\sup_{1 \leq t \leq 2} \int f\left(L(L^{-1}x-t\Gamma(s))\right) \, ds\right| 
=\left|\sup_{1 \leq t \leq 2} \int f_L(L^{-1}x-t\Gamma(s)) \, ds\right| \nonumber \\
&=&M(f_L)(L^{-1}x),~~f_L(x)=f(Lx),\nonumber
\end{eqnarray}

Thus, using change of variable formula with $det(J_L)=1,$ we get
\begin{eqnarray}
\|\tilde{M}f(x)\|_{L^q}^q&=&\|Mf_L(L^{-1}x)\|_{L^q}^q=\|Mf_L(x)\|_{L^q}^q \nonumber \\
&\leq& \|M\|_{L^p \rightarrow L^q}^q\|f_L(x)\|_{L^p}^q=\|M\|_{L^p \rightarrow L^q}^q\|f(x)\|_{L^p}^q. \nonumber
\end{eqnarray}
Hence, we get conclusion of the Lemma.
\end{proof}

\section{Decomposition and proof of Theorem \ref{T03.1}}
In this section, we will prove our main Theorem \ref{T03.1}. The main idea of the proves of the Theorem \ref{T03.1} is to decompose the maximal operator $\mathbb{M}f$ into a family of maximal operators $\mathbb{M}_kf.$  In this paper, we use the well-known decomposition, the decomposition into parts where the curvature is approximately $2^{-k}$. In this connection, Iosevich \cite{IOSE} was able  to decompose his operator into parts where $s \sim 2^{-k},$ because he was interested only in finite type curves. Such a decomposition is not appropriate when one considers infinitely flat curves. Clearly these two methods are equivalent for finite type curves, and indeed, for such curves we shall obtain essentially the same estimates as Iosevich. Thus, for non-negative $f$, we have
$$\mathbb{M}_{\sigma}f(x) \leq C\sum_{k\geq 0} 2^{-k\sigma} M_kf(x)$$
where
$$M_kf(x)=\sup_{1 \leq t \leq 2} \left|\int_{I_k}f(x-t\Gamma(s)) \, ds\right|,$$
with $I_k=[s_{k+1}, s_k]$ is the interval where $2^{-k-1} \leq g^{''}(s) \leq 2^{-k}.$ By scaling and a change of variable, we see that $M_k$ will have the same $L^q \longrightarrow L^q$ operator norms as the operator $\tilde{M}_k$ , given by
$$\tilde{M}_k f(x)=|I_k|\sup_{1 \leq t \leq 2}\left|\int_{0}^1f\left(x-t\left(s+s_{k+1}|I_k|^{-1},\rho_k^{-1} g(s|I_k|+s_{k+1})+\rho_k^{-1}\right)\right) \, ds\right|$$
for any fixed $\rho_k>0.$ Suppose there exist constants $\eta_k, \beta_k, \alpha_k$ such that
\begin{eqnarray} \label{04.1}
&&\left(s+s_{k+1}|I_k|^{-1}, \, \rho_k^{-1} g(s|I_k|+s_{k+1})+\rho_k^{-1}\right) \nonumber \\
&&=\left(s+\alpha_k, \, \omega(s)+\beta_k+\eta_k(s+\alpha_k)\right)
\end{eqnarray}

for some curve $\omega$ of the type dealt with in Lemma \ref{Lma2}, and with $C^5$ norms bounded independent of $k$. In principal, $\omega$ ought to have a $k$ subscript, but the idea is that since
its $C^5$ norm can be bounded independent of $k$, it is effectively a constant curve (almost) independent of $k$.

By Lemma \ref{Lma4.1} (with $Lv=(v_1, v_2-\eta_k v_1)$), we would then have that $\tilde{M_k}$ has operator norms controlled by $|I_k|$ times those of the maximal function corresponding to the curve
$$(s+\alpha_k, \omega(s)+\beta_k)$$
which by Lemma \ref{Lma2} has operator norms bounded by a constant times
$$B|I_k||\beta_k|^{1/q},$$
for $ (\frac{1}{p}, \frac{1}{q}) \in \left[(\frac{1}{p}, \frac{1}{q}) :(\frac{1}{p}, \frac{1}{q}) \in \triangle \setminus \{P, Q\} \right]  \cap \left[(\frac{1}{p}, \frac{1}{q}) :q > 2 \right],$ with $B$ independent of $k,$ if $|\beta_k| \geq max \{|\alpha_k|, 2\}.$ 

The question is, how can we arrange that \eqref{04.1} be true ? For completeness, 
we shall briefly give the answer to this question. An explicit proof can also be found in \cite{GM}. 

A Taylor expansion of $g(s)$ about $s=s_{k+1}$ suggests writing
$$g(s|I_k|+s_{k+1})=g(s_{k+1})+s|I_k| g'(s_{k+1}) +\tilde{\omega}(s)$$
where $\tilde{\omega}(0)=\tilde{\omega}'(0)=0$ and $\tilde{\omega}$ and its derivatives inherit their monotonicity from $g.$

Notice that for $s \in [0,1], ~\tilde{\omega}^{''}(s)\sim 2^{-k}|I_k|^2.$ Hence it is natural to choose 
\begin{eqnarray} \label{04.2}
\omega(s)=2^k |I_k|^{-2}\{g(s|I_k|+s_{k+1})-g(s_{k+1})-s|I_k| g'(s_{k+1})\}
\end{eqnarray}

so that $1/2 \leq \omega^{''}(s) \leq 1.$ Notice that $\omega$ inherits the monotonicity properties of $\tilde{\omega}.$

Now, we see that 
$$\rho_k^{-1} g(s|I_k|+s_{k+1})+\rho_k^{-1}=\rho_k^{-1}\{2^{-k} |I_k|^2 \omega(s)+g(s_{k+1})+1+s|I_k| g'(s_{k+1})\}$$
and so choosing $\rho_k=2^{-k}|I_k|^2,$ we see that \eqref{04.1} will hold so long as 
$$\alpha_k=s_{k+1}|I_k|^{-1},~ \eta_k=2^k|I_k|^{-1} \, g'(s_{k+1})$$ and 
$$\beta_k=2^k|I_k|^{-2}\{g(s_{k+1})+1-s_{k+1} g'(s_{k+1})\}.$$

Now, we will show that $\beta_k,~ \alpha_k$ satisfies the conditions of Lemma \ref{Lma2}, i.e., $|\beta_k|\geq max\{|\alpha_k, 2|\}.$ Recall that, $I_k=[s_{k+1}, s_k]$ is the interval where $2^{-k-1} \geq g^{''} \geq 2^{-k}.$ Hence, by mean value theorem, for all but finitely many $k,$ we have $$s_{k+1} \, g'(s_{k+1}) \leq 1/2.$$
So that, we have $|\beta_k| \geq 2^{k-1} \, |I_k|^{-2},$ for all but finitely many $k.$ Thus, for all but finitely many $k,$ we have $|\beta_k|\geq max\{|\alpha_k, 2|\}$, since 
\begin{eqnarray*}
&&2^{k-1}|I_k|^{-2} \geq 2 \Leftrightarrow |I_k|^2 \leq 2^{k-2} \\
\mbox{ and } && \\ && 2^{k-1}|I_k|^{-2} \geq  \frac{s_{k+1}}{|I_k|} \Leftrightarrow |I_k| \leq \frac{2^{k-1}}{s_{k+1}},
\end{eqnarray*} 

both hold for all but finitely many $k.$

We notice that, $|\beta_k|\leq 2^k|I_k|^{-2}.$ Assuming we can show that $\omega^{(r)}(s)\leq C,$ for $r=1,2,3,4,5$ and $C$ independent of $k,$ then we have that $M_k$ has $L^p-L^q$ operator norms bounded by 
$$B|\beta_k|^{1/q} \, |I_k| \leq B 2^{k/q}|I_k|^{1-2/q}$$ so long as $ (\frac{1}{p}, \frac{1}{q}) \in \left[(\frac{1}{p}, \frac{1}{q}) :(\frac{1}{p}, \frac{1}{q}) \in \triangle \setminus \{P, Q\} \right]  \cap \left[(\frac{1}{p}, \frac{1}{q}) :q > 2 \right]$ (by Lemma \ref{Lma2}).
This would then complete the proof of Theorem \ref{T03.1}, as
\begin{eqnarray}
\|\mathbb{M}_{\sigma}\|_{L^p-L^q}&\leq& B\sum_{k \geq 0} 2^{-k\sigma} \|M_k\|_{L^p-L^q} \nonumber \\
&\leq& B \sum_{k \geq 0}2^{-k(\sigma-1/q)} |I_k|^{1-2/q} < \infty \nonumber
\end{eqnarray}

if $\sigma > 1/q.$ We return now to those derivative estimates for $\omega.$ We shall show that $\omega$ has $C^5$ norms independent of $k.$ The point here is that although $\omega$ depends on $k,$ we shall show that it has $C^5$ norm independent of $k$ and satisfies all the conditions of Lemma \ref{Lma2}.

From \eqref{04.2}, we notice that because the second derivative of $\omega$ is small, and because $\omega(0)=\omega'(0)=0,$ we have that $\omega'(s)\leq s$ and $\omega(s)\leq s^2/2.$ 

We have 
$$\omega^{(r)}(s)=2^k |I_k|^{r-2} \, g^{(r)}(s|I_k|+s_{k+1})$$
for $r=3,4,5.$ Given the monotonicity of $g^{(r)},$ and given that $2^{-k}=g^{''}(s_k),$  we clearly require only that 
\begin{eqnarray} \label{04.3}
|I_k|^{r-2} \frac{g^{(r)}(s_k)}{g^{('')}(s_k)} \leq C,
\end{eqnarray}
where $C$ independent of $k.$ This is where we shall use those conditions on $\log g^{''}$ and $\log g^{'''}.$ The proof of \eqref{04.3} can be found in \cite{GM}. We briefly give the outline of the proof for the case $r=3.$ For the cases $r=4$ and $r=5$ will follow from our assumption that $\log g{''}$ and $\log g^{'''}$ are concave respectively. To proof \eqref{04.3} for the case $r=3,$ it is enough to show that $$|I_k|\sim \frac{g^{''}(s_k)}{g^{'''}(s_k)}.$$ 
Now the mean value theorem gives us the following:
$$|I_k|=(g^{''})^{-1}(2^{-k})-(g^{''})^{-1}(2^{-k-1})=2^{-k-1} \frac{1}{g^{'''}(z)}$$ for some $z \in (s_{k+1}, s_k),$ and we get
\begin{eqnarray} \label{004.4}
c\frac{g^{''}(s_k)}{g^{'''}(s_k)} \leq |I_k| \leq C \frac{g^{''}(s_{k_1})}{g^{'''}(s_{k+1})}.
\end{eqnarray}

As our assumptions imply that $\frac{g^{''}(s)}{g^{'''}(s)}$ is monotonic increasing, and as $s_{k+1}< s_k,$ the right-hand side of \eqref{004.4} is bounded by a constant times the left-hand side giving the stated result, that is $|I_k|\sim \frac{g^{''}(s_k)}{g^{'''}(s_k)}.$ This completes the proof for the case $r=3$ and hence the proof of Theorem \ref{T03.1}.

\section{The case $\sigma=1/q$ and sharpness of results}
In this section, we will see that a power $\sigma < 1/q$ of the curvature will not guarantee $L^p-L^q,~ (q> p)$ boundedness for our maximal function. For example, consider $\Gamma(s)=(s, s^m+1)$ and the test function 
$$f(x)=\eta(x) \, |x_2|^{-1/q}$$ for some $q>p,$ where $\eta \in C_c^{\infty}$ is a smooth bump function such that $\eta(x)=1$ if $|x|\leq 1,$ and $\eta(x)=0$ if $|x|\geq 2.$ Notice that $f \in L^r$ for all $r < q,$ and hence $f$ is in $L^p.$

Consider only points $x$ that lie within distance $1/2$ of the origin, so that $|x-x_2 \Gamma(s)| \leq 1$ for all $s \in [0,1].$ Then we have for $1 \leq x_2 \leq 2$,
\begin{eqnarray}
\mathbb{M}_{\sigma}f(x) &\geq& C \int_0^1f(x-x_2\Gamma(s)) s^{(m-2) \sigma} \, ds \nonumber \\
&\geq& C \int_0^1|x_2|^{-1/q} s^{(m-2)\sigma-m/q} \, ds
\end{eqnarray}

which diverges unless $(m-2)\sigma-m/q > -1.$ So, if we wish to choose $\sigma$ independent of $m \geq 2$, we clearly need $\sigma \geq 1/q.$

It is more interesting to ask what happens when we consider $\sigma=1/q < 1/2$. In this case we know by our earlier work that
$$\|\mathbb{M}_{\sigma}f\|_{L^p} \leq B \, \|f\|_{L^p} \, \sum_{k\geq 0}|I_k|^{1-2/q},$$
so if $|I_k|$ decays sufficiently fast, then we will have that $\mathbb{M}$ is $L^p-L^q$ bounded. For example, we have the following corollary.
\begin{Corollary} \label{Cor5.1}
Suppose $\Gamma$ and $g$ are as in Theorem \ref{T03.1}, and suppose in addition that there exists an $\epsilon > 0$ such that 
$$|\log ~ g^{''}(s)|^{1+\epsilon}\left(\frac{g^{''}(s)}{g^{'''}(s)}\right)$$
is bounded.

Then we have $|I_k|\leq Bk^{-1-\epsilon},$ and for $ (\frac{1}{p}, \frac{1}{q}) \in \left[(\frac{1}{p}, \frac{1}{q}) :(\frac{1}{p}, \frac{1}{q}) \in \triangle \setminus \{P, Q\} \right]  \cap \left[(\frac{1}{p}, \frac{1}{q}) :q = \sigma^{-1} > 2(\epsilon^{-1}+1) \right]$, the maximal function of Theorem \ref{T03.1} is bounded on $L^p-L^q.$

\end{Corollary}

\begin{proof}
The proof of Corollary \ref{Cor5.1} is a matter of simple calculus, working with the function $\gamma : \mathbb{R}^+ \longrightarrow [0, 1],$ defined implicitly by
$$\log ~ g^{''}(\gamma(x))=-(x^{-1/{\epsilon}}) ~ \log ~ 2.$$
This is chosen so that
$$\gamma(k^{-\epsilon}) = s_{k},$$ where $s_k$ is as before, defined by $g^{''}(s_k)=2^{-k}.$ We have that $|I_k|=\gamma(k^{-\epsilon}) - \gamma((k+1)^{-\epsilon}),$ which is $(k^{-\epsilon}-(k+1)^{-\epsilon}) \, \gamma'(z),$ for some $z \in (k, k+1).$ We ask that $\gamma'$ be uniformly bounded, this being the hypothesis of Corollary \ref{Cor5.1}, and hence $|I_k|\leq C k^{-\epsilon-1},$ from which the desired result follows easily. 
\end{proof}

\begin{Remark}
Notice that, the Corollary \ref{Cor5.1} applies to the family of flat curves $g(s)=\exp(-s^{-N}),$ for $N \geq 1,$ with $\epsilon=1/N.$
\end{Remark}

\noindent
{\bf Acknowledgements:} 
I wish to thank the Harish-Chandra Research institute, the Dept. of Atomic Energy, Govt. of India, for providing excellent research facility.

%\newpage

\end{document}